\documentclass[11pt]{article}
\usepackage[margin=1in]{geometry}
\usepackage{amsfonts}
\usepackage{pgf}
\usepackage{amsmath, amsthm, amssymb, graphicx}

\newtheorem{theorem}{Theorem}[section]
\newtheorem{lemma}[theorem]{Lemma}
\newtheorem{proposition}[theorem]{Proposition}

\newenvironment{prf}[1]{\trivlist
\item[\hskip
\labelsep{\it #1.\hspace*{.3em}}]}{
\endtrivlist}

\newtheorem{predefinition}[theorem]{Definition}
\newenvironment{definition}{\begin{predefinition}\rm}{\end{predefinition}}
\newtheorem{preremark}[theorem]{Remark}
\newenvironment{remark}{\begin{preremark}\rm}{\end{preremark}}
\newtheorem{prenotation}[theorem]{Notation}

\newtheorem{preexample}[theorem]{Example}

\newtheorem{preclaim}[theorem]{Claim}

\newtheorem{prequestion}[theorem]{Question}

 \makeatletter
\def\emppsubsection{\@startsection{subsection}{2}{\z@}{-3.25ex plus -1ex minus -.2ex}{-1em}{\bf}}

\makeatother

\newcommand{\ZZ}{{\mathbb Z}}
\newcommand{\FF}{{\mathbb F}}
\newcommand{\Jac}{\textrm{Jac}}

\author{Shawn Farnell and Rachel Pries}
\title{Families of Artin-Schreier curves with Cartier-Manin matrix of constant rank}
\date{\today}

\begin{document}

\maketitle

\begin{abstract}
Let $k$ be an algebraically closed field of characteristic $p > 0$. 
Every Artin-Schreier $k$-curve $X$ has an equation of the form $y^p-y=f(x)$ for some $f(x) \in k(x)$ 
such that $p$ does not divide the least common multiple $L$ of the orders of the poles of $f(x)$. 
Under the condition that $p \equiv 1 \bmod L$, Zhu proved that the Newton polygon of the $L$-function of $X$ is determined by
the Hodge polygon of $f(x)$.  In particular, the Newton polygon depends only on the
orders of the poles of $f(x)$ and not on the location of the poles or otherwise on the coefficients of $f(x)$.  
In this paper, we prove an analogous result about the $a$-number of the $p$-torsion group scheme of the Jacobian of $X$, providing the first non-trivial examples of families of Jacobians with constant $a$-number.  Equivalently, we consider the semi-linear Cartier operator on the sheaf of regular 1-forms of $X$ and provide the first non-trivial examples of families of curves whose Cartier-Manin matrix has constant rank.
\end{abstract}

\section{Introduction}

Suppose $k$ is an algebraically closed field of characteristic $p > 0$ and 
$X$ is an Artin-Schreier $k$-curve, namely a smooth projective connected $k$-curve which is a 
$\ZZ/p$-Galois cover of the projective line. 
Studying the $p$-power torsion of the Jacobian of $X$ is simultaneously feasible and challenging.
For example, zeta functions of Artin-Schreier curves over finite fields are analyzed in 
\cite{LauderWan2, LauderWan, NartSadornil, wan:weil}. 
Newton polygons of Artin-Schreier curves are the focus of the papers \cite{Blache,Blache2,BlacheFerardZhu,SholtenZhu, GeerVlugt}.

Every Artin-Schreier $k$-curve $X$ has an equation of the form $y^p-y=f(x)$ for some non-constant rational function $f(x) \in k(x)$ 
such that $p$ does not divide the order of any of the poles of $f(x)$.
The genus of $X$ depends only on the orders of the poles of $f(x)$.
Let $m+1$ denote the number of poles of $f(x)$ and let 
$d_0, \ldots, d_{m}$ denote the orders of the poles. 
By the Riemann-Hurwitz formula, the genus of $X$ is $g_X =D(p-1)/2$ where 
$D=\sum_{j=0}^{m} (d_{j}+1) - 2$.
By definition, the $p$-rank of the Jacobian ${\rm Jac}(X)$ of $X$ 
is the dimension $s_X$ of ${\rm Hom}_{\mathbb F_p}(\mu_{p},\text{Jac}(X)[p])$
where $\mu_p$ denotes the kernel of Frobenius morphism $F$ on the multiplicative group scheme $\mathbb{G}_{m}$.
The $p$-rank also equals the length of the slope $0$ portion of the Newton polygon.  
For an Artin-Schreier curve $X$, the $p$-rank $s_X$ equals $m(p-1)$ by the Deuring-Shafarevich formula, and thus depends only on the number of poles of $f(x).$ 

In most cases, the Newton polygon of $X$ is not determined by the orders of the poles of $f(x)$.
One exception was found by Zhu: 
let $L$ denote the least common multiple of the orders of the poles of $f(x)$; 
under the condition that $p \equiv 1 \bmod L$,
the Newton polygon of $X$, shrunk by the factor $p-1$ in the horizontal and vertical direction, 
equals the Hodge polygon of $f(x)$ \cite[Corollary 1.3]{Zhu:newtonoverhodge}, see Remark
\ref{Rzhu}.  
In particular, this means that the Newton polygon depends only on the
orders of the poles of $f(x)$ and not on the location of the poles or otherwise on 
the coefficients of $f(x)$.
In this paper, we prove an analogous result about the $a$-number of the Jacobian ${\rm Jac}(X)$
or, equivalently, about the rank of the Cartier-Manin matrix of $X$.

The $a$-number is an invariant of the $p$-torsion group scheme ${\rm Jac}(X)[p]$.
Specifically, if $\alpha_{p}$ denotes the kernel of Frobenius on the additive group $\mathbb{G}_{a}$, 
then the $a$-number of (the Jacobian of) $X$ is $a_X = {\rm dim}_{k}\text{Hom}(\alpha_{p},\text{Jac}(X)[p])$. 
It equals the dimension of the intersection of ${\rm Ker}(F)$ and ${\rm Ker}(V)$ on the 
Dieudonn\'e module of ${\rm Jac}(X)[p]$, where $V$ is the Verschiebung morphism.
The $a$-number and the Newton polygon place constraints upon each other, 
but do not determine each other, see e.g., \cite{Harashita1, Harashita2}.

The $a$-number is the co-rank of the Cartier-Manin matrix, which is the matrix for the modified Cartier operator on the sheaf of regular 1-forms of $X$. The modified Cartier operator is the $1/p$-linear map ${\mathcal C}: H^{0}(X,\Omega_X^1) \rightarrow H^{0}(X,\Omega_X^1)$ taking exact 1-forms to zero and satisfying ${\mathcal C}(f^{p-1}df) = df$.  In other words, the $a$-number equals the dimension of the kernel of ${\mathcal C}$ on  $H^{0}(X,\Omega_X^1)$.

In this paper, under the condition $p \equiv 1 \bmod L$, 
we prove that the $a$-number of $X$ depends only on the orders of poles of $f(x)$ and not on the 
location of the poles or otherwise on the coefficients of $f(x)$ (see section \ref{Smainresult}).

\begin{theorem} \label{Tintrothm}
Let $X$ be an Artin-Schreier curve with equation $y^p-y=f(x)$, with $f(x) \in k(x)$. 
Suppose $f(x)$ has $m+1$ poles, with orders $d_0, \dots, d_{m}$, and 
let $L={\rm LCM}(d_0, \ldots, d_{m})$. 
If $p \equiv 1 \bmod L$, then the $a$-number of $X$ is 
\[a_X = \sum_{j=0}^{m} a_{j}, \ \text{ where } \ 
a_j = \begin{cases} (p-1)d_j/4 & \text{if $d_j$ even,}\\
  (p-1)(d_j-1)(d_j+1)/4d_j & \text{if $d_j$ odd.} \end{cases}\]
\end{theorem}

To our knowledge, Theorem \ref{Tintrothm} provides the first non-trivial examples of families of Jacobians with constant $a$-number when $p \geq 3$.
When $p=2$, the main result of \cite{ElkinPries} is that the Ekedahl-Oort type (and $a$-number) of an Artin-Schreier curve 
depend only on the orders of the poles of $f(x)$.
For arbitrary $p$, it is easy to construct families of Jacobians with $a_X=0$ (ordinary) or $a_X=1$ (almost ordinary) and  
a family of Jacobians with $a_X=2$ is constructed in \cite[Corollary 4]{GlassPries}.

For fixed $p$, the families in Theorem \ref{Tintrothm} 
occur for every genus $g$ which is a multiple of $(p-1)/2$.
The $a$-number of each curve in the family is roughly half of the genus. 
Using \cite[Theorem 1.1 (2)]{P/Z}, the dimension of the family can be computed to be $\sum_{i=0}^{m}(d_j+1)-3 = 2g/(p-1) - 1$.

Other results about $a$-numbers of curves can be found in \cite{Ekedahl,Elkin}. 
We end the paper with some open questions motivated from this work.

The second author was partially supported by NSF grant DMS-1101712.

\section{Background}

\subsection{Artin-Schreier curves} \label{Sartsch}

Let $k$ be an algebraically closed field of characteristic $p>0$. 
A {\it curve} in this paper is a smooth projective connected $k$-curve.
An {\it Artin-Schreier curve} is a curve $X$ which admits a $\ZZ/p$-Galois cover of the projective line.
Letting $x$ be a coordinate on the projective line, every Artin-Schreier curve has an equation of the form $y^p-y=f(x)$ for some 
non-constant rational function $f(x)\in k(x)$.
By Artin-Schreier theory, after a change of variables, $f(x)$ can be chosen such that 
$p$ does not divide the order of any pole of $f(x)$.
We assume that this is the case throughout the paper.

Let ${\mathbb B} \subset \mathbb{P}^1(k)$ be the set of poles of $f(x)$ and suppose $\#{\mathbb B}=m+1$.
We can assume that $\infty \in {\mathbb B}$ after a fractional linear transformation.
We choose an ordering of the poles ${\mathbb B}=\{b_0, \ldots, b_m\}$ such that $b_0=\infty$.
For $b_j \in {\mathbb B}$, let $d_j$ be the order of the pole of $f(x)$ at $b_j$.
Let $x-e_j$ be a uniformizer at $b_j$ for $1 \leq j \leq m$.   
Let $x_0 = x$ and let $x_j=(x-e_j)^{-1}$ if $1 \leq j \leq m$. 
The partial fraction decomposition of $f(x)$ has the form:
\[f(x) = f_0(x) + \sum_{j=1}^{m} f_{j}\left(\frac{1}{x-e_j}\right)
=\sum_{j=0}^{m} f_{j}\left(x_j\right),\] where
$f_{j}(x_j) \in k[x_j]$ is a polynomial of degree $d_j$ for $0 \leq j \leq m$ and
$f_j(x)$ has no constant term for $1 \leq j \leq m$.
Let $u_j \in k^{\times}$ be the leading coefficient of $f_{j}(x_j)$. 

\subsection{The genus and $p$-rank of an Artin-Schreier curve}\label{genussection}

The genus of a curve $X$ is the dimension of the vector space $H^0(X,\Omega^1_X)$ of regular 1-forms. 
By the Riemann-Hurwitz formula \cite[Proposition VI.4.1]{Stich}, the genus of an Artin-Schreier curve $X:y^p-y=f(x)$ 
where $f(x)$ has $m+1$ poles with prime-to-$p$ orders $d_0, \ldots, d_j$ as described in Section \ref{Sartsch} is
\[g_X = D(p-1)/2 {\rm \ where \ } D= -2 + \sum_{j=0}^{m}(d_{j}+1).\]

Given a smooth projective $k$-curve $X$ of genus $g$, 
let $\text{Jac}(X)[p]$ denote the $p$-torsion group scheme of the Jacobian of $X$. 
Let $\mu_p$ be the kernel of Frobenius on the multiplicative group $\mathbb{G}_m$. 
The {\it $p$-rank} of $X$ is $s_X = \text{dim}_{\mathbb{F}_p}\text{Hom}(\mu_p,\text{Jac}(X)[p])$.
The number of $p$-torsion points of $\Jac(X)(k)$ satisfies 
$\#\Jac(X)\left[p\right](k) = p^{s_X}$.
The $p$-rank of a curve satisfies the inequality $0 \leq s_X \leq g$.
By a special case of the Deuring-Shafarevich formula, see \cite[Theorem 4.2]{Subrao} or \cite{Crew},
if $X$ is an Artin-Schreier curve with equation $y^p - y = f(x)$ as described above, then the $p$-rank of $X$ is $s_X = m(p-1)$.

\subsection{The $a$-number}

Let $\alpha_p$ be the kernel of Frobenius on the additive group $\mathbb{G}_a$.  
The {\it $a$-number} of $X$ is $a_X = \text{dim}_{k}\text{Hom}(\alpha_p,\text{Jac}(X)[p])$.
Equivalently, the $a$-number is the dimension of ${\rm Ker}(F) \cap {\rm Ker}(V)$
on the Dieudonn\'e module of $\Jac(X)[p]$.
The $a$-number also equals the dimension of ${\rm Ker}(V)$ on $H^0(X, \Omega^1_X)$ \cite[5.2.8]{LO}.
By definition, $0 \leq a_X + s_X \leq g$.

The $a$-number is an invariant of the $p$-torsion group scheme ${\rm Jac}(X)[p]$.
In some cases, it gives information about ${\rm Jac}(X)$ as well.
If $a_X=g$, then $\Jac(X)$ is isomorphic to a product of supersingular elliptic curves.
If $s_X < g$, then $a_X > 0$.
This can be used to show that the number of factors appearing
in the decomposition of ${\rm Jac}(X)$ into simple principally polarized abelian varieties is at most $s_X + a_X$.

\begin{remark}
In \cite{Pr:RM}, formulas are given for the $a$-number of an Artin-Schreier curve
when $f(x)$ is a monomial $x^d$ with $p \nmid d.$ If $p \equiv 1 \bmod d$, then the main result of this paper extends \cite[Corollary 3.3]{Pr:RM} to all Artin-Schreier curves $X: y^p-y=f(x)$ having the property that the orders of the poles of $f(x)$ divide $p-1.$ If $p \not\equiv 1 \bmod d$, let $h_b \in [0,p-1]$ be the integer such that $h_b \equiv (-1-b)d^{-1} \bmod p$. By \cite[Remark 3.4]{Pr:RM}, the $a$-number of $X:y^p-y=x^d$ is given by \[a_X = \sum_{b=0}^{d-2} \text{min}\left(h_b,p-\left\lceil (p+1+bp)/d\right\rceil\right).\]
\end{remark}

\subsection{The Cartier operator and the $a$-number}

The (modified) Cartier operator $\mathcal{C}$ is the semi-linear map 
${\mathcal{C}}: H^0(X,\Omega^1_X) \rightarrow H^0(X,\Omega^1_X)$ with the following properties: ${\mathcal{C}}(\omega_1 + \omega_2) = {\mathcal{C}}(\omega_1) + {\mathcal{C}}(\omega_2)$;
${\mathcal{C}}(f^p\omega) = f{\mathcal{C}}(\omega)$; and
\[{\mathcal{C}}(f^{n-1}df) = 
   \begin{cases}
    df& \text{if $n=p$},\\
    0&  \text{if $1 \leq n < p$}.
   \end{cases}\]
Suppose $\beta = \{\omega_1,\ldots,\omega_g\}$ is a basis for $H^0(X,\Omega^1_X)$. 
For each $\omega_j$, let $m_{i,j} \in k$ be such that \[{\mathcal{C}}\left(\omega_j\right) = \sum_{i=1}^g m_{i,j}\:\omega_i.\] 
The $g \times g$-matrix $M=(m_{i,j})$ is the (modified) Cartier-Manin matrix 
and it gives the action of the (modified) Cartier operator. 
The Cartier-Manin matrix is $\tilde{M}=(m_{i,j}^p)$; it is the matrix for the (unmodified) Cartier operator, see \cite{Yui}.
The action of $V$ is the same as the action of the (unmodified) Cartier operator on $H^0(X, \Omega^1)$, see \cite{Cartier}, 
and so the $a$-number satisfies $a_X = g_X - {\rm rank}(\tilde{M}) =g_X- {\rm rank}(M)$.
At the risk of confusion, we drop the word modified in the rest of the paper. 
 
\section{The $a$-number of a family of Artin-Schreier curves} \label{MyANumber}

\subsection{Regular 1-forms on an Artin-Schreier curve}

Let $X$ be an Artin-Schreier curve as described in Section \ref{Sartsch}. By \cite[Lemma 1]{Sullivan}, 
a basis for $H^{0}(X,\Omega_X^1)$ is given by $W=\cup_{j=0}^{r} W_j$ where
\begin{align}
&W_{0} = \left\{x^{b}y^{r}dx \big| r,b\geq0 \text{ and } rd_{0}+bp \leq (p-1)(d_{0}-1)-2\right\}, \text{ and } \nonumber \\
&W_{j} = \left\{x_{j}^{b}y^{r}dx \big| r\geq0, b\geq 1, \text{ and } rd_{j}+bp \leq (p-1)(d_j+1)\right\} \text{ if $1 \leq j \leq m$}. \nonumber
\end{align}

There is a slight difference between the cases $j=0$ and $1 \leq j \leq m$.  
This is in some way unavoidable as can be seen from the formula for the $p$-rank. 
To shorten the exposition, we let $\epsilon_j=-1$ if $j=0$ and $\epsilon_j=1$ if $1 \leq j \leq m$.
Note that $\# W_j = (d_j+\epsilon_j)(p-1)/2.$

We define an ordering $\prec$ on the basis $W$. Define $x_{i}^{b_1}y^{r_1}dx \prec x_{j}^{b_2}y^{r_2}dx$ if $r_1 < r_2$, or if
$r_1 = r_2$ and $i < j$, or if $r_1 = r_2$, $i=j$ and $b_1 < b_2$.

\subsection{Action of the Cartier operator}

Consider the action of the Cartier operator on $H^0(X, \Omega^1_X)$.
In general,
\begin{align}
{\mathcal C}\left(x_{j}^{b}y^{r}dx\right) &= {\mathcal C}\left(x_{j}^{b}\left(y^p - f(x)\right)^{r}dx\right). \nonumber
\end{align}

To simplify notation, let $\tau=(\tau_{-1},\dots,\tau_{m})$ denote a tuple of length $m+2$
whose entries are non-negative integers and let $|\tau|=\sum_{j=-1}^{m} \tau_j$.
Using the extended binomial theorem, we see that
\[\left(y^p - f(x)\right)^{r} = 
\sum_{\tau, |\tau|=r} c_{\tau} y^{p{\tau_{-1}}} f_{0}^{\tau_0}(x) f_{1}^{\tau_1}\left(x_{1}\right) \cdots f_{m}^{\tau_m}\left(x_m\right),\] 
where  
\[\displaystyle c_{\tau} = (-1)^{r-\tau_{-1}}\binom{r}{\tau_{-1},\dots,\tau_m}.\]
So,
\begin{equation} \label{Ecartier}
{\mathcal C}\left(x_{j}^{b}y^{r}dx\right)
= \sum_{\tau, |\tau|=r} c_{\tau} y^{\tau_{-1}} {\mathcal C}\left(x_{j}^{b} f_{0}^{\tau_0}(x) f_{1}^{\tau_1}\left(x_{1}\right) \cdots f_{m}^{\tau_m}\left(x_{m} \right) dx \right).
\end{equation}
One can check that 
\begin{equation}\label{E2}
{\mathcal C}\left(x_j^{ap+\epsilon_j}dx\right) = x_j^{a+\epsilon_j}dx.
\end{equation}

\subsection{An assumption on the orders of the poles} \label{Spole}

Let $L={\rm LCM}(d_0, \ldots, d_m)$.
From now on, we assume that $p \equiv 1 \bmod L$; 
in other words, the order $d_j$ of the $j$th pole of $f(x)$ divides $p-1$
and we define $\gamma_j = (p-1)/d_j$ for $0 \leq j \leq m$.
Under this condition, we prove a result about the $a$-number of the Jacobian of $X$ which is analogous to the following result of Zhu:

\begin{remark} \label{Rzhu}
Suppose $f(x) \in \FF_q(x)$ for some power $q$ of $p$ and let $N_s=\#X(\FF_{q^s})$ for $s \in {\mathbb N}$.
Since $X$ is a smooth projective curve, the zeta function of $X$ is a rational function of the form:
\[Z_X(u):={\rm exp}(\sum_{s=1}^\infty \frac{N_su^s}{s}) = \frac{L_X(u)}{(1-u)(1-qu)},\]
where the $L$-function $L_X(u) \in \ZZ[u]$ is a polynomial of degree $2g$.
Under the condition $p \equiv 1 \bmod L$, Zhu proved that 
the Newton polygon of $L_X(u)$
(shrunk by a factor of $p-1$ in the horizontal and vertical direction) equals the Hodge polygon of $f(x)$ \cite[Corollary 1.3]{Zhu:newtonoverhodge}.  
The Hodge polygon has slopes of $0$ and $1$ each occurring with multiplicity $m$ and 
slopes $\left\{1/d_j, \ldots, (d_j-1)/d_j\right\}$ for $0 \leq j \leq m$.
In particular, this means that the Newton polygon depends only on the
orders of the poles of $f(x)$ and not on the location of the poles or otherwise on 
the coefficients of $f(x)$.
\end{remark}

Under the condition $p \equiv 1 \bmod L$, for $0 \leq j \leq m$, the 1-forms $x_{j}^{b}y^{r}dx \in W_j$ are in bijection 
with ordered pairs $(b,r)$ of integers in the closed triangle bounded by $r=0$, $b=(1+\epsilon_j)/2$, and $r=(p-2+ \epsilon_j \gamma_j)-\gamma_{j}b$.

\subsection{Linearly independent columns of the Cartier-Manin matrix}

In this section, we define a subset $H \subset W$ and show that the columns of the 
Cartier-Manin matrix associated with elements of $H$ are linearly independent.
This gives a lower bound on the rank of the Cartier-Manin matrix, and thus an upper bound on the 
$a$-number.

Recall that $\epsilon_j=-1$ if $j=0$ and $\epsilon_j=1$ if $1 \leq j \leq m$.
We partition the 1-forms in $W_j$ into two subsets:
\[H_{j} = \left\{ x_j^by^rdx \in W_{j} \  \Big| \ r \geq (b-\epsilon_j)\gamma_{j} \right\},\]
and the set-theoretic complement
\[A_{j} = W_{j} - H_{j}.\]

\begin{figure}[ht]
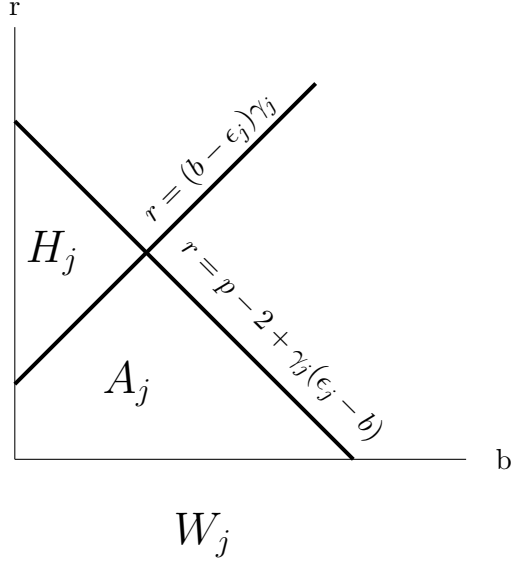

\begin{center}
\begin{pgfpicture}{-1cm}{-1cm}{6cm}{6.5cm}
\pgfline{\pgfxy(0,0)}{\pgfxy(6,0)}
\pgfline{\pgfxy(0,0)}{\pgfxy(0,5.75)}
\begin{pgfscope}
\pgfsetlinewidth{1.5pt}
\pgfline{\pgfxy(0,4.5)}{\pgfxy(4.5,0)}
\end{pgfscope}
\pgfputlabelrotated{.74}{\pgfxy(0,4.7)}{\pgfxy(4.5,0.2)}{6pt}{\pgfbox[center,base]
{$r=p-2+\gamma_{j}(\epsilon_j-b)$}}
\begin{pgfscope}
\pgfsetlinewidth{1.5pt}
\pgfline{\pgfxy(0,1)}{\pgfxy(4,5)}
\end{pgfscope}
\pgfputlabelrotated{0.7}{\pgfxy(0,1)}{\pgfxy(4,5)}{6pt}{\pgfbox[center,base] 
{$r=(b-\epsilon_j)\gamma_{j}$}}
\pgfputat{\pgfxy(0.5,2.75)}{\pgfbox[center,center]{\LARGE $H_j$}}
\pgfputat{\pgfxy(1.5,1)}{\pgfbox[center,center]{\LARGE $A_j$}}
\pgfputat{\pgfxy(2.5,-1)}{\pgfbox[center,center]{\LARGE $W_j$}}
\pgfputat{\pgfxy(6.5,0)}{\pgfbox[center,center]{b}}
\pgfputat{\pgfxy(0,6)}{\pgfbox[center,center]{r}}
\end{pgfpicture}
\end{center}

\caption{The subsets $H_{j}$ and $A_{j}$ of $W_j.$}
\label{DiffsPic}

\end{figure}

Let $H = \cup_{j=0}^{m} H_j$ and $A = \cup_{j=0}^{m} A_j.$

\begin{definition} \label{Dkeyterm}
If $\omega = x_j^by^rdx \in H_j$, the {\it key term} $\kappa({\mathcal C}(\omega))$ of ${\mathcal C}(\omega)$ is
the $1$-form $x_j^{b}y^{r-(b-\epsilon_j)\gamma_{j}}dx$.
\end{definition}

\begin{lemma} \label{GeneralClaim1}
If $\omega \in H$, the coefficient of $\kappa({\mathcal C}(\omega))$ is non-zero
in ${\mathcal C}(\omega)$.
\end{lemma}

\begin{proof}
Suppose $\omega \in H_j$ for some $0 \leq j \leq m$.  
The claim is that, if $r \geq (b-\epsilon_j)\gamma_{j}$, then the coefficient of the 1-form $x_{j}^{b}y^{r-(b-\epsilon_j)\gamma_j}dx$ in ${\mathcal C}(x_{j}^{b}y^{r}dx)$ is non-zero.
Consider the tuple $\tau$ given by $\tau_{-1}=r-(b-\epsilon_j)\gamma_j$, $\tau_j = (b-\epsilon_j)\gamma_j$, and $\tau_i = 0$ for all $i \not\in \{-1,j\}$.
If $r \geq (b-\epsilon_j)\gamma_j$, by Equation \eqref{Ecartier}, the following term appears in ${\mathcal C}(x_{j}^{b}y^{r}dx)$:
\begin{equation} \label{E3}
c_{\tau}y^{r-(b-\epsilon_j)\gamma_j}{\mathcal C}\left(x_{j}^{b}f_{j}^{(b-\epsilon_j)\gamma_j} \left(x_j\right) dx \right).
\end{equation}
Because $\text{deg}_{x_j}(x_j^bf_{j}^{(b-\epsilon_j)\gamma_{j}}(x_j)) = (b-\epsilon_j)p + \epsilon_j$, we see from \eqref{E2} that $c_{\tau} u_{j}^{(b-\epsilon_j)\gamma_j/p} x_{j}^{b}y^{r-(b-\epsilon_j)\gamma_j}dx$ appears in \eqref{E3}.

The coefficient $c_{\tau}$ in Equation \eqref{E3} is nonzero because $r \leq p-2$ for all $\omega \in H$. Also, $u_{j} \neq 0$ as it is the leading coefficient of $f_j(x_j)$. 
This term is canceled by no others. To see this, notice that the coefficient of $x_{j}^{b}y^{r-(b-\epsilon_j)\gamma_j}dx$ in Equation \eqref{Ecartier}
is zero unless $\tau_{-1} = r - (b-\epsilon_j)\gamma_{j}$ and $\tau_{j} \geq (b-\epsilon_j)\gamma_{j}$.
\end{proof}

The next lemma shows that the coefficient of $\kappa({\mathcal C}(\omega))$ is zero in ${\mathcal C}(\omega')$
for any 1-form $\omega' \in W$ which is smaller than $\omega$. 

\begin{lemma} \label{GeneralClaim3}
If $\omega \in H$ and $\omega' \in W$ with $\omega' \prec \omega$, 
then the coefficient of $\kappa({\mathcal C}(\omega))$ is zero in ${\mathcal C}(\omega')$.
\end{lemma}

\begin{proof}
Write $\omega'=x_k^By^{R}dx$ and recall the calculation: 
\begin{equation} \label{Ecartier2}
{\mathcal C}\left(x_k^By^{R}dx\right) = \sum_{\tau, |\tau| = R} c_{\tau} y^{\tau_{-1}} {\mathcal C}\left(x_k^B f_{0}^{\tau_0}(x) f_{1}^{\tau_1}(x_1) \cdots f_{m}^{\tau_m}(x_{m})\right).
\end{equation}

{\bf Case 1:}  Suppose $\omega = x^{b}y^{r} dx \in H_{0}$. The claim is that the coefficient $c_{\omega}$ of $\kappa({\mathcal C}(\omega))=x^{b}y^{r-(b+1)\gamma_{0}}dx$ in Equation \eqref{Ecartier2} is zero for any $\omega' \prec x^by^rdx$. The coefficient $c_{\omega}$ will be zero unless $\tau_{-1} = r-(b+1)\gamma_{0}$. This gives the restriction that $\tau_0 \leq R-(r-(b+1)\gamma_{0})$.

If $k=0$, $c_{\omega}$ will be zero unless $\tau_{0}d_{0} + B \geq (b+1)p - 1$. 
Combining these inequalities yields that 
\[R-r \geq (b-B)/d_{0}.\]
Because both $b$ and $B$ are less than $d_{0}-2$, $c_{\omega}$ is non-zero only if $R > r$ or if $R=r$ and $B \geq b$. 

If $k \neq 0$, the coefficient $c_{\omega}$ of $x^{b}y^{r-(b+1)\gamma_{0}}dx$ in Equation \eqref{Ecartier2} will be zero 
unless $\tau_{0}d_{0} - B \geq (b+1)p-1$.
Combining the given inequalities shows that \[R-r \geq (b+B)/d_{0}.\]
As $B>0$, this shows that $c_{\omega}$ is non-zero only if $R>r$. 
In both cases, $\omega'=x_k^By^Rdx \not \prec \omega=x^by^rdx$.

{\bf Case 2:} Suppose $\omega \in H_j$ for some $1 \leq j \leq m$. 
The claim is that the coefficient $c_{\omega}$ of the 1-form $x_{j}^{b}y^{r-(b-1)\gamma_{j}}dx$ in ${\mathcal C}(\omega')$ is zero for any $\omega' \prec x_{j}^{b}y^rdx$.
The coefficient $c_{\omega}$ is non-zero only if $\tau_{-1} = r-(b-1)\gamma_j$.  This gives the restriction that $\tau_j \leq R-(r-(b-1)\gamma_{j})$. 

If $k \neq j$, then $c_{\omega}$ is non-zero only if $\tau_{j}d_{j} \geq (b-1)p+1$ and so
\[R-r \geq b/d_{j}.\]
As $b>0$, $c_{\omega}$ is non-zero only if $R>r$.

If $k =j$, the coefficient $c_{\omega}$ is non-zero only if $\tau_{j}d_{j}+B \geq (b-1)p+1$ which yields that
\[R - r \geq (b-B)/d_j.\]
Since $b$ and $B$ are both bounded by $d_j$, this is only satisfied if $R>r$ or if $R=r$ and $B \geq b$, 
in other words, only if $\omega'=x_k^By^Rdx \not \prec \omega=x^by^rdx$.
\end{proof}

\begin{proposition} \label{Pindep}
The columns of the Cartier-Manin matrix $M$ corresponding to the 1-forms in $H$ are linearly 
independent.
\end{proposition}

\begin{proof}
This follows from Lemmas \ref{GeneralClaim1} and \ref{GeneralClaim3} since the key terms $\kappa({\mathcal C}(\omega))$ yield pivots of $M$ for $\omega \in H$.  
\end{proof}

\subsection{Linearly dependent columns of the Cartier-Manin matrix}

In this section, we prove that the columns of the Cartier-Manin matrix associated with the 1-forms in $A$ do not contribute to the rank of the Cartier-Manin matrix, because they are linearly dependent on the columns associated with the 1-forms in $H$. 

For fixed $j$ and $r$, let $B$ vary and consider the ordered pair $(B,R)$ of exponents in $\kappa({\mathcal C}(x_{j}^{B}y^{r}dx)$.
The points $(B,R)$ lie on a line of slope $-\gamma_j$,
specifically the line $R=r+\epsilon_j \gamma_j - \gamma_j B$,
where $\epsilon_{j}=-1$ if $j=0$ and $\epsilon_{j}=1$ if $1 \leq j \leq m$.
For $0 \leq j \leq m$ and $r \leq (p-2)/2$, let 
\[Z_{j,r} = \left\{ x_j^By^Rdx \in W_j \ \Big| \ R=r + \epsilon_j \gamma_{j}-\gamma_{j}B \right\}.\]
Note that $Z_{0,r}$ is empty if $0 \leq r < \gamma_{0}.$ Let
\[Y_{j,r} = \begin{cases}  
                  \cup_{\ell=\gamma_{0}}^{r} Z_{0,\ell} & \text{if $j=0$,} \\
                  \cup_{\ell=0}^{r} Z_{j,\ell} & \text{if $1 \leq j \leq m$.} 
                  \end{cases}\]

\begin{lemma}\label{GeneralClaim4}
Suppose $\eta = x_{j}^{b}y^{r}dx \in W_j$ for some $0 \leq j \leq m$ with $r \leq (p-2)/2$.
Then ${\mathcal C}(\eta) \in \mathrm{span}(Y_{i,r} \mid 0 \leq i \leq m)$. 
\end{lemma}

\begin{proof}
Fix $\sigma \in W_i$ with $0 \leq i \leq m$ and 
let $c_\sigma$ denote the coefficient of $\sigma$ in ${\mathcal C}(\eta)$.
It suffices to show that $\sigma \in Y_{i,r}$ whenever $c_\sigma \not = 0$.
Write $\sigma = x_i^{B}y^{R}dx$.  By Equation \eqref{Ecartier}, $c_\sigma=0$ unless $\tau_{-1} = R$. 
This gives that $\tau_{i} \leq r-R$. 
If $R \geq r+\epsilon_i\gamma_{i} - \gamma_{i}B+1$ then $\tau_{i} \leq \gamma_{i}B-\epsilon_i\gamma_{i}-1$. 
The degree of $x_i$ in $x_j^{b}f_{i}^{\tau_{i}}(x_i)$ satisfies 
\begin{align}
\text{deg}_{x_i}\left(x_j^{b}f_{i}^{\tau_{i}}(x_i)\right) &\leq b+\tau_{i}d_{i} \nonumber \\
&\leq b + \left(\gamma_{i}B-\epsilon_i\gamma_{i}-1\right)d_{i} \nonumber \\
&= (B-\epsilon_i)p-B+b+\epsilon_i-d_{i}. \nonumber
\end{align}
By the definition of $W_{i}$, if $i=0$ then $b \leq d_{0}-2$ and $B \geq 0$, 
and if $1 \leq i \leq m$ then $b \leq d_i$ and $B \geq 1$. So, $\text{deg}_{x_i}(x_j^{b}f_{i}^{\tau_{i}}(x_i)) < (B-\epsilon_i)p+\epsilon_i$. 
Thus, $c_\sigma=0$ when $R > r+\epsilon_i\gamma_{i} - \gamma_{i}B$.
\end{proof}

\begin{lemma}\label{GeneralClaim6}
Suppose $r \leq (p-2)/2$ and $0 \leq i \leq m$. Every element of $Y_{i,r}$ is a key term of ${\mathcal C}(\omega)$ for some $\omega \in H_i$.
\end{lemma}

\begin{proof}
Let $x_{i}^{B}y^{R}dx \in Y_{i,r}$. Define $\omega=x_{i}^{B}y^{\rho} dx$ where 
$\rho = R- \epsilon_i \gamma_{i}+\gamma_{i}B$.
It suffices to show that $\omega \in H_i$, since $\kappa({\mathcal C}(\omega)) = x_{i}^{B}y^{R}dx$.
If $x_{i}^{B}y^{R}dx \in Y_{i,r}$ then $R \leq r+\epsilon_i\gamma_{i}-\gamma_{i}B$, so 
$\rho \leq r.$ The 1-form $x_{i}^{B}y^{\rho}dx$ is in $H_i$ because $B \geq 0$, 
and $- \epsilon_i\gamma_{i}+\gamma_{i}B \leq \rho \leq (p-2)/2$.
\end{proof}

\begin{lemma} \label{GeneralClaim7}
If $\eta \in A$, 
then ${\mathcal C}(\eta)$ is contained in 
$\mathrm{span}\left\{{\mathcal C}(\omega)\  | \ \omega \in H \right\}$.
\end{lemma}

\begin{proof}
Write $\eta = x_{j}^{b}y^{r}dx$ for some $0 \leq j \leq m$.  
Since $\eta \in A$, $r \leq (p-2)/2$.
By Lemma \ref{GeneralClaim4}, ${\mathcal C}(\eta) \in \text{span}(Y_{i,r} \mid 0 \leq i \leq m)$. 
By Lemma \ref{GeneralClaim6}, $\mathcal{C}(\eta) \in \mathrm{span}\{\kappa({\mathcal C}(\omega)) \mid \omega \in H\}.$ 
Let $\omega^* = x_j^By^Rdx$ be the largest 1-form in $H$ for which the coefficient of $\kappa({\mathcal C}(\omega^*))$ in ${\mathcal C}(\eta)$ is non-zero. From the proof of Lemma \ref{GeneralClaim6}, we see that $R < (p-2)/2.$
Let $\nu \in k^{\times}$ be such that the coefficient of $\kappa({\mathcal C}(\omega^*))$ is zero 
in ${\mathcal C}(\eta) - \nu {\mathcal C}(\omega^*)$.
If $\tau$ is a monomial in ${\mathcal C}(\omega^*)$, then $\tau = \kappa({\mathcal C}(\omega^{**}))$ for some $\omega^{**} \in H.$ Lemma \ref{GeneralClaim3} implies that $\omega^{**} \prec \omega^*.$ Therefore, the terms in ${\mathcal C}(\eta) - \nu {\mathcal C}(\omega^*)$ are key terms of ${\mathcal C}(\omega^{**})$ for $\omega^{**} \prec \omega^*.$
Repeating this process shows that ${\mathcal C}(\eta)$ can be written as 
a linear combination $\sum_{\omega \in H} \nu_{\omega}{\mathcal C}(\omega)$.
\end{proof}

\subsection{Main result} \label{Smainresult}

\begin{theorem}
Let $X$ be an Artin-Schreier curve with equation $y^p-y=f(x)$, with $f(x) \in k(x)$. 
Suppose $f(x)$ has $m+1$ poles, with orders $d_0, \dots, d_{m}$ and 
let $L={\rm LCM}(d_0, \ldots, d_{m})$. 
If $p \equiv 1 \bmod L$, then the $a$-number of $X$ is  
\[a_X = \sum_{j=0}^{m} a_{j}, \text{ where } \nonumber 
a_j = \begin{cases} (p-1)d_j/4 & \text{if $d_j$ even,}\\
  (p-1)(d_j-1)(d_j+1)/4d_j & \text{if $d_j$ odd.} \end{cases} \nonumber\]
\end{theorem}

\begin{proof}
By Proposition \ref{Pindep} and Lemma \ref{GeneralClaim7}, 
the rank of the Cartier-Manin matrix is equal to $\sum_{j=0}^m \#H_{j}$.
Since $a = g - \text{rank}(M)$ and $g=\#W$, this implies $a = \sum_{j=0}^{m} (\#W_{j} - \#H_{j})$. 
It thus suffices to show that $\#W_j - \#H_j = a_j$ for the value of $a_j$ 
as stated for $0 \leq j \leq m$.

Recall that $\#W_{j}=(p-1)(d_{j}+\epsilon_j)/2$.
We will count the ordered pairs $(b,r)$ corresponding to $x^{b}y^{r}dx \in H_{j}$. The lines $r=p-2+\epsilon_j\gamma_{j}-\gamma_{j}b$ and $r=\gamma_{j}b-\epsilon_j\gamma_{j}$ intersect at $b=d_{j}/2 + \epsilon_j - 1/2\gamma_{j}$. The largest value of $b$ appearing in $H_{j}$ is 
\[b' = \begin{cases} d_{j}/2 + \epsilon_j - 1 & \text{if $d_{j}$ is even,} \\ d_{j}/2 + \epsilon_j - 1/2 & \text{if $d_{j}$ is odd.} \end{cases}\]
Let $b_j = 0$ if $j=0$ and $b_j = 1$ if $j \neq 0$.  Then
\begin{align}
a_j &= \# W_j - \#H_j \nonumber \\
  &= (p-1)(d_{j}+\epsilon_j)/2 - \sum_{b_j}^{b'} \left(p-1+2\epsilon_j\gamma_{j}-2\gamma_{j}b\right) \nonumber \\
  &= (p-1)(d_{j}+\epsilon_j)/2 - \left(p-1+2\epsilon_j\gamma_{j}\right)(b'-b_j+1) + 2\gamma_{j}b'\left(b'+1\right)/2 \nonumber \\
  &= \begin{cases} (p-1)d_{j}/4 & \text{if $d_{j}$ even,} \\ (p-1)(d_{j}-1)(d_{j}+1)/4d_j & \text{if $d_{j}$ odd.} \end{cases} \nonumber
\end{align}
\end{proof}

\subsection{Open questions}

Here are two questions that emerge from this work:

{\bf Question 1:} 
Under the condition $p \equiv 1 \bmod L$, are the Ekedahl-Oort type and the Dieudonn\'e module of the Jacobian of the Artin-Schreier curve $X:y^p-y=f(x)$ determined by the orders of the poles of $f(x)$?

{\bf Question 2:} What are other families of curves for which the $p$-rank, Newton polygon, $a$-number, and Ekedahl-Oort type of the fibres of the family are constant?  

For example, when $p=2$, the Ekedahl-Oort type (and $2$-rank and $a$-number) of an Artin-Schreier (hyperelliptic) curve depend only on the orders of the poles of $f(x)$ \cite{ElkinPries}.

\bibliographystyle{plain}
\bibliography{farnellpries}

\vspace{.2in}
\noindent Shawn Farnell \\
Department of Mathematics, Kenyon College\\
Gambier, Ohio 43022, USA\\
farnells@kenyon.edu

\vspace{.2in}
\noindent Rachel Pries \\
Department of Mathematics, Colorado State University \\
Fort Collins, CO 80523-1874, USA\\ 
pries@math.colostate.edu
\end{document}